\documentclass{article}[12pt]
\usepackage{graphicx} 
\usepackage[utf8]{inputenc}
\usepackage{authblk}
\usepackage{setspace}
\graphicspath{ {./figures/} }
\usepackage{subcaption}
\usepackage{amssymb,amsthm,amsmath}

\usepackage{color,xcolor}

\usepackage{hyperref}
\usepackage{lineno}

\usepackage{fullpage}

\newtheorem{theorem}{\bf Theorem}[section]
\newtheorem{proposition}[theorem]{\bf Proposition}
\newtheorem{lemma}[theorem]{\bf Lemma}
\newtheorem{corollary}[theorem]{\bf Corollary}

\newtheorem{observation}[theorem]{\bf Observation}

\newtheorem{conjecture}{\bf Conjecture}

\theoremstyle{plain}
\newtheorem*{claim}{\bf Claim}

\title{Matchings in hypergraphs via Ore-degree conditions}

\author{J\'ozsef Balogh\thanks{Department of Mathematics, University of Illinois Urbana-Champaign, Urbana, IL, USA, and Extremal Combinatorics and Probability Group (ECOPRO), Institute for Basic Science (IBS), Daejeon, South Korea. Email: \texttt{jobal@illinois.edu}. Partially supported by NSF grants RTG DMS-1937241, FRG DMS-2152488, (UIUC  Campus Research Board Award RB26026), the Simons Fellowship, Simons Collaboration grant [SFI-MPS-TSM-00013107, JB], and the Institute for Basic Science (IBS-R029-C4).}
\qquad Cory Palmer\thanks{Department of Mathematical Sciences, University of Montana. E-mail: \texttt{cory.palmer@umontana.edu}. Supported by NSF grant DMS-2503179 and a grant from the Simons Foundation [SFI-MPS-TSM-00013277, CP].}
\qquad Ghaffar Raeisi\thanks{Department of Mathematical Sciences, Shahrekord University, Shahrekord, P.O.Box 115, Iran, and School of Mathematics, Institute for Research in Fundamental Sciences (IPM), P.O. Box: 19395--5746, Tehran, Iran, Email: \texttt{g.raeisi@ipm.ir}. This work is based upon research funded by Iran National Science Foundation (INSF) under
project No. 40400962 and also supported in part by a grant from IPM No. 1404050317.}
}

\begin{document}

\maketitle
\begin{abstract}
Let $\mathcal{H} \subseteq \binom{[n]}{r}$ be an $r$-uniform hypergraph on vertex set $[n] = \{1,2,\dots, n\}$.
  For an $r$-set of vertices $S \subseteq [n]$, 
 the {\it degree} of $S$ is defined as $\deg(S)=\sum_{v \in S}\deg(v)$ and the minimum of $\deg(S)$ over all non-edge $r$-subsets $S \not \in E(\mathcal{H})$ of $V({\cal H})$ is the {\it Ore-degree} of ${\cal H}$, denoted by ${\sigma_r}({\cal H})$. 
We prove several Ore-degree results about existence of matchings in hypergraphs:

    (1) For  $n\geq 2r+2$, if ${\cal  H}$ is an intersecting  $r$-uniform hypergraph on $n$ vertices, then $\sigma_r({\cal H})\leq r{n-2 \choose r-2}$, and  there is equality only  when  ${\cal  H}$ is a  $1$-star.

    (2) For  $r\geq 3$ 
and $n\geq 4r^2$, if 
is a non-trivial intersecting $r$-uniform hypergraph on $n$ vertices, then $\sigma_r({\cal H})\leq r\left({n-2 \choose r-2}-{n-r-2 \choose r-2}\right)$.

(3) For $s\geq 2$ and $n\geq 3r^2(s-1)$, if ${\cal  H}$ is an $r$-uniform hypergraph on $n$ vertices and $\sigma_r({\cal H})>r\left({n-1 \choose r-1}-{n-s \choose r-1}\right)$, then ${\cal  H}$ contains $s$ pairwise disjoint edges.

\end{abstract}

\section{ Introduction}

An {\it $r$-uniform hypergraph} (\emph{$r$-graph}) ${\cal H}$ is a 
family of $r$-subsets (the \emph{edges} of $\mathcal{H}$) of a \emph{vertex} set $V$.
For a given hypergraph ${\cal H}$, denote the vertex set by $V(\mathcal{H})$ and, for simplicity, we usually suppose that $V({\cal H})$ is $[n]:=\{1,2,\dots, n\}$. We identify $\mathcal{H}$ with its set of edges and so by $e \in \mathcal{H}$ we mean that $e$ is an edge of $\mathcal{H}$. Thus $|\mathcal{H}|$ is the number of edges in $\mathcal{H}$. We often write $\mathcal{H} \subseteq \binom{[n]}{r}$ to define an $r$-uniform hypergraph $\mathcal{H}$ on vertex set $[n]$.

The {\it degree} of a vertex $v\in V({\cal H})$ is defined as the number of distinct edges containing $v$ and the {\it minimum degree} of ${\cal H}$ (resp.\ {\it maximum degree} of ${\cal H}$), $\delta({\cal H})$ (resp. $\Delta({\cal H})$), is the minimum (resp. the maximum) of $\deg(v)$ over all vertices $v\in V({\cal H})$. For an $r$-set $S \subseteq V({\cal H})$, 
the {\it degree of} $S$ is defined as $\deg(S)=\sum_{v \in S}\deg(v)$ and the minimum of $\deg(S)$ over all non-edge $r$-subsets $S \not \in E(\mathcal{H})$ of $V({\cal H})$ is the \emph{Ore-degree} of $\mathcal{H}$ and is denoted ${\sigma_r}({\cal H})$. 
For an ordinary graph $H$, we recover the standard definition of Ore degree:  $\sigma_2(H)=\min_{uv\notin E(H)}(\deg(u)+\deg(v))$. 

A
{\it matching} in a hypergraph ${\cal H}$ is a set of pairwise disjoint edges and the {\it matching number} 
 is the size of the largest matching. 

\medskip
An $r$-uniform hypergraph ${\cal H}$ is  {\it intersecting}, if every pair of its edges has a non-empty intersection.
The celebrated theorem of Erd\H{o}s, Ko, and Rado~\cite{EKR} states that when $n\geq 2r$, every intersecting $r$-uniform hypergraph on $[n]$ has at most ${n-1 \choose r-1}$ edges and moreover, for $n >2r$, the extremal hypergraph is unique up to isomorphism and it consists of all the $r$-subsets of $[n]$ that contain a fixed vertex $x\in [n]$. 
We call such hypergraph a {\it 1-star}.
The Erd\H{o}s--Ko--Rado theorem is widely regarded as the cornerstone of extremal set theory and has many interesting applications and generalizations, see \cite{FG} for a survey.   Huang and Zhao~\cite{huang} gave the following minimum degree version of the Erd\H{o}s--Ko--Rado theorem.

\begin{theorem}[Huang and Zhao~\cite{huang}]\label{hoangzho}
Let $n\geq 2r+1$. If every vertex in an $n$-vertex $r$-uniform hypergraph ${\cal  H}$ has degree at least ${n-2 \choose r-2}$, then either ${\cal  H}$ is a 1-star or is non-intersecting (i.e., $\mathcal{H}$ contains two disjoint edges).
\end{theorem}

 The proof is based on the analysis of eigenvalues of the Kneser graph and a lemma from combinatorial geometry. It is worth noting that recently, Frankl and Tokushige \cite{FT}, gave a combinatorial proof of this result for $n\geq 3r$. In fact, using their method, one can actually get a proof of Theorem~\ref{hoangzho} for $n\geq 2r+3$.

\medskip
Hilton and Milner~\cite{HM} proved a strong stability extension of the Erd\H{o}s--Ko--Rado theorem. An intersecting hypergraph ${\cal H}$ is  \emph{trivial} if it is a subhypergraph of a $1$-star.  The Hilton--Milner theorem states that for $ n>2r$, every non-trivial intersecting $r$-uniform hypergraph on $n$ vertices has at most $\binom{n-1}{r-1} - \binom{n-r-1}{r-1} + 1$ members. A corresponding extremal hypergraph ${\cal HM}_{n,r}$ consists of an $r$-edge $S$ and all $r$-subsets of $[n]$ containing a fixed element $x \notin S$ and at least one element of $S$. A degree version of the Hilton and Milner theorem was also studied by Frankl,  Han, Hao, and Yi~\cite{FHHY} who proved that if $n \geq \Omega(r^2)$ and ${\cal H}$ is a non-trivial intersecting $r$-uniform hypergraph on $n$ vertices, then $\delta({\cal H}) \leq \delta({\cal HM}_{n,r})={n-2 \choose r-2}-{n-r-2 \choose r-2}$.

\medskip
In general, the problem of finding a maximum matching in a hypergraph is an important problem in extremal combinatorics and has many applications in various different areas of mathematics and  computer science. 
Although the graph matching problem is fairly well-understood, and solvable in polynomial time, many problems related to hypergraph matching appear to be  difficult and remain unsolved. For example, the hypergraph matching problem is known to be NP-hard even for
3-uniform hypergraphs. One of the most basic open questions in this area was raised in 1965 by Erd\H{o}s \cite{erdos}, who asked to determine the maximum possible number of edges that can appear in any $r$-uniform hypergraph ${\cal H}$ with matching number less than $s$.
He conjectured that extremal constructions for this problem are one of the following two hypergraphs. The first is a clique consisting of all $r$-subsets of $rs-1$ vertices, which obviously has matching number $s-1$. The second is an $r$-uniform hypergraph on $n$ vertices containing all edges that intersect a fixed set of $s-1$ vertices, which also forces the matching number to be $s-1$.

\begin{conjecture}[Erd\H os~\cite{erdos}]\label{conj}
Let $n\geq rs-1$.
If ${\cal H}$ is an $n$-vertex $r$-uniform hypergraph such that 
\[
|{\cal H}|>\max\left\{{rs-1 \choose r}, {n \choose r}-{n-(s-1) \choose r}\right\},
\]
then ${\cal H}$ contains $s$ pairwise disjoint edges.
\end{conjecture}

When a hypergraph does not contain two edges whose intersection is an empty set, then it is intersecting, hence 
for $s=2$,
Conjecture~\ref{conj} is  equivalent to the Erd\H{o}s--Ko--Rado theorem. 
The graph case (i.e., when $r=2$) was verified separately by Erd\H{o}s and Gallai~\cite{erdos-gallai}. When $r=3$, Frankl, R\"{o}dl, and Ruci\'{n}ski~\cite{FRR}, proved the conjecture for $n\geq 4s$, and {\L}uczak and Mieczkowska~\cite{LM} proved it for every sufficiently large $s$. Recently, Frankl~\cite{F} proved the conjecture for $r=3$. For arbitrary $r$, Bollob\'{a}s, Daykin, and Erd\H{o}s~\cite{BDE} proved the conjecture for $n>2r^3(s-1)$, and  Huang, Loh, and Sudakov~\cite{HLS} for  $n\geq 3r^2(s-1)$. Recently, Frankl~\cite{F2} further improved it as follows.

\begin{theorem}[Frankl~\cite{F2}]\label{F3}
  Let $n\geq (2s-1)r-(s-1)$. If ${\cal H}$ is an $r$-uniform hypergraph on $n$ vertices such that $|{\cal H}|>{n \choose r}-{n-(s-1) \choose r},$ then ${\cal H}$ contains $s$ pairwise disjoint edges. 
\end{theorem}

Bollobás, Daykin, and Erd\H{o}s \cite{BDE} also considered a minimum degree version of Conjecture~\ref{conj} and showed that if $n>2r^3(s-1)$ and ${\cal H}$ is an  $r$-uniform hypergraph on $n$ vertices with $\delta({\cal H})\geq {n-1 \choose r-1}-{n-s\choose r-1}$, then $\mathcal{H}$ has matching number $s$.
Recently, Huang, and Zhao~\cite{huang} improved this result by reducing the bound on $n$ to $n\geq 3r^2s$.

\medskip

In this paper, we extend these results by considering hypergraphs satisfying an Ore-degree condition replacing the minimum degree condition. 
There has been much effort to extend results on graphs from minimum degree conditions to Ore-degree conditions. The first result of this type is the famous theorem of Ore~\cite{ore}: If for every pair of non-adjacent vertices $u$ and $v$ in an $n$-vertex graph $G$, we have $\deg(u)+\deg(v)\geq n$, then $G$ is Hamiltonian. Other results of this type include results on  $k$-ordered Hamiltonian graphs~\cite{fad}, equitable colorings~\cite{equ}, an Ore version of Brooks' theorem~\cite{brooks}, $H$-linked graphs~\cite{Fe} and partitions into  monochromatic cycles~\cite{Barat}.

The main results of the paper are as follows. The first two can be viewed as Ore-degree analogues of the Erd\H os--Ko--Rado theorem and the Hilton--Milner theorem, respectively.

\begin{theorem}\label{erdosko}
Let $r\geq 3$ and  $n\geq 2r+2$. 
If ${\cal  H}\subseteq {[n]\choose r}$ is an intersecting hypergraph, then $\sigma_r({\cal H})\leq r{n-2 \choose r-2}$. There is equality only  when  ${\cal  H}$ is a  $1$-star.
\end{theorem}

\begin{theorem}\label{main3}
Let $r \geq 3$ and $n \geq 4r^2$.
If ${\cal  H}\subseteq {[n]\choose r}$ 
is a non-trivial intersecting hypergraph, then $\sigma_r({\cal H})\leq r \left({n-2 \choose r-2}-{n-r-2 \choose r-2}\right)$.
\
\end{theorem}

The bounds on $\sigma_r({\cal H})$ in both Theorems~\ref{erdosko} and \ref{main3} are best possible as given by  $1$-stars and the Hilton--Milner construction ${\cal HM}_{n,r}$. 
We also prove an  Ore-degree version of the Erd\H{o}s' Matching Conjecture.

\begin{theorem}\label{mainn}
Let $s\geq 2$ and $n\geq 3r^2(s-1)$. If ${\cal  H}$ is an $r$-uniform hypergraph on $n$ vertices and $\sigma_r({\cal H})>r \left({n-1 \choose r-1}-{n-s \choose r-1}\right)$, then ${\cal  H}$ contains $s$ pairwise disjoint edges.
\end{theorem}

Note that the $r$-uniform hypergraph on $n$ vertices containing all edges intersecting a fixed set of $s-1$ vertices has matching number $s-1$ and $\sigma_r({\cal H})=r \left({n-1 \choose r-1}-{n-s \choose r-1}\right)$. Thus, the bound given for $\sigma_r({\cal H})$ in Theorem~\ref{mainn} is best possible.

\medskip
W also extend the Ore-degree version of the Erd\H{o}s Matching Conjecture to properly edge-colored hypergraphs. If we assign colors to the edges of a hypergraph $\mathcal{H}$, we say that 
 ${\cal H}$ is {\it properly edge-colored} if no two incident edges receive the same color. Thus, the 
 edges incident to a vertex $v$ all receive distinct colors.
A \emph{$s$-rainbow matching} in a family of properly edge-colored hypergraphs ${\cal H}_{1},\ldots, {\cal H}_{s}$ is a collection of pairwise disjoint edges $e_{1},\ldots,e_{s}$ of distinct colors such that $e_{1}\in {\cal H}_{1},\ldots,e_{s}\in {\cal H}_{s}$. 
Huang, Li, and Wang~\cite{HLW} proved a rainbow version of the Erd\H{o}s Matching Conjecture, as follows.

\begin{theorem}[Huang, Li, and Wang~\cite{HLW}]\label{HLW}
Let $n\geq 3r^{2}s$. If ${\cal H}_{1}, \dots, {\cal H}_{s}$ are properly edge-colored $r$-uniform hypergraphs on $n$ vertices such that for each $i$, $|{\cal H}_{i}|>\binom{n}{r}-\binom{n-s+1}{r}$, then there exists an $s$-rainbow matching in ${\cal H}_{1}, \dots, {\cal H}_{s}$.
\end{theorem}

We prove the following Ore-degree version of Theorem~\ref{HLW}.

\begin{theorem}\label{properlycolored}
Let $r\geq 3$ and $n>3r^2s$. If ${\cal H}_{1}, \dots, {\cal H}_{s}$ are properly edge-colored $r$-uniform hypergraphs on $n$ vertices such that $\sigma_r({\cal H}_i)> r\left({n-1 \choose r-1}-{n-s \choose r-1}\right)$, then there exists an $s$-rainbow matching in ${\cal H}_{1}, \dots, {\cal H}_{s}$.
\end{theorem}

The notion of an intersecting hypergraph can be generalized to pairs of hypergraphs in the following way. Two hypergraphs ${\cal A}$ and ${\cal B}$  are {\it cross-intersecting} if for every two edges $A \in {\cal A}$  and $B \in {\cal B}$, their intersection $A\cap B$ is non-empty. Pyber~\cite{Pyber} showed that for $n\geq 2r$, if ${\cal A}, {\cal B} \subseteq \binom{[n]}{r}$ are cross-intersecting, then $|{\cal A}||{\cal B}|\leq {n-1 \choose r-1}^2$. Huang and Zhao~\cite{huang} proved the following degree version of this result.

\begin{theorem}[Huang and Zhao~\cite{huang}]\label{huang}
Let $n\geq 2r+1$. If ${\cal A}, {\cal B} \subseteq \binom{[n]}{r}$ are cross-intersecting, then $\delta({\cal A})\delta({\cal B})\leq {n-2 \choose r-2}^2$.
\end{theorem}

 Our final result is an Ore-degree version of Theorem~\ref{huang}.

\begin{theorem}\label{main2} Let $n\geq 4r^2$. If ${\cal A},{\cal B} \subseteq \binom{[n]}{r}$ are cross-intersecting, then $\sigma_r({\cal A})\sigma_r({\cal B})\leq r^2{n-2 \choose r-2}^2$. 
\end{theorem}

Theorem~\ref{main2} is sharp as evidenced by taking $\mathcal{A},\mathcal{B}$ to be copies of the same $1$-star.

\medskip

\noindent
{\bf Remark.}
While preparing our manuscript, independently, two papers were posted on arXiv which provide different (and longer) proofs of a slightly weaker version of Theorem~\ref{erdosko}.
In particular,
Huang and Rao~\cite{huangrao}, extending work of Frankl and Wang~\cite{franklwang} proved that if $\cal H$ is an intersecting $r$-uniform hypergraph on $n$ vertices, then the $(2r+1)$-st largest degree of $\cal H$ is at most $\binom{n-1}{r-2}$. As every set of $2r$ vertices spans a non-edge $r$-set in an intersecting hypergraph, this implies that for $n\ge 4r$, we have 
$\sigma_r({\cal H})\leq r{n-2 \choose r-2}$, i.e., a slightly weaker result than our Theorem~\ref{erdosko}.

\section{Tools}

In this section we list several lemmas and theorems that we use in our proofs. The Ore-degree $\sigma_r$ is monotone as adding an edge to an $r$-uniform hypergraph cannot decrease $\deg(S)$ for any $r$-set $S$.
Thus, we will use the following observation implicitly.

\begin{observation}
For an $r$-uniform hypergraph $\mathcal{H}$, we have
    $\sigma_r(\mathcal{H}') \leq \sigma_r(\mathcal{H})$ when $\mathcal{H}' \subseteq \mathcal{H}$.
\end{observation}

We use the following standard inequalities repeatedly:
\begin{equation}\label{handy-ineq}
c \binom{a-1}{b-1}> \binom{a}{b} - \binom{a-c}{b} > c \binom{a-c}{b-1},    
\end{equation}
which can be seen by estimating the number of $b$-subsets of an $a$-set that intersect a fixed $c$-subset in at least one element. 

If the  degrees of the vertices  in a hypergraph $\mathcal{H}$ are all equal, then we call $\mathcal{H}$ a \emph{regular} hypergraph.
We prove the following lemma that establishes a tight lower bound on the size of an $n$-vertex $r$-uniform hypergraph ${\cal  H}$ in terms of $\sigma_r({\cal H})$.

\begin{lemma}\label{lem}
Let ${\cal  H} $ be an $r$-uniform hypergraph on $n$ vertices. Then  $|{\cal H}|\geq \frac{n}{r^2}\cdot \sigma_r({\cal H})$, with equality only if $\mathcal{H}$ is regular.
\end{lemma}

\begin{proof}
For each of the $\binom{n}{r}-|{\cal H}|$ non-edge $r$-sets $S \notin {\cal H}$ we have $\sum_{v \in S} \deg(v) \geq \sigma_r({\cal H}).$ Summing over all non-edge $r$-sets, we obtain
   \[
   \sum_{S \in  \binom{[n]}{r} \setminus {\cal H}} \sum_{v \in S} \deg(v) \geq \sigma_r({\cal H})  \left(\binom{n}{r}-|{\cal H}|\right).
   \]
Since the number of non-edge $r$-sets containing a vertex $v$ is $\binom{n-1}{r-1}-\deg(v)$, the left-hand side counts each vertex degree $\deg(v)$ exactly $\binom{n-1}{r-1}-\deg(v)$ times. Thus,
   \[
   \sum_{v \in V({\cal H})} \deg(v) \left(\binom{n-1}{r-1}-\deg(v)\right) \geq \sigma_r({\cal H}) \left(\binom{n}{r}-|{\cal H}|\right).
   \]
Now, $\sum_{v \in V(\mathcal{H})} \deg(v)= r|{\cal H}|$, hence 
   \[
   r |{\cal H}| \binom{n-1}{r-1}-\sum_{v\in V(\mathcal{H})} \deg(v)^2 \geq \sigma_r({\cal H}) \left(\binom{n}{r}-|{\cal H}|\right).
   \]
Applying the Cauchy-Schwartz inequality gives $$\sum \deg(v)^2 \geq \frac{1}{n}\left(\sum \deg(v)\right)^2=\frac{r^2 |{\cal H}|^2}{n} .$$ Therefore,
  
  \[
  r|{\cal H}| \binom{n-1}{r-1}-\frac{r^2 |{\cal H}|^2}{n} \geq \sigma_r({\cal H})  \left(\binom{n}{r}-|{\cal H}|\right).
  \]
Equivalently,
\[
0\geq \frac{r^2}{n}|{\cal H}|^2-\left(\sigma_r({\cal H})+r\binom{n-1}{r-1}\right)|{\cal H}|+\sigma_r({\cal H})\binom{n}{r}
 =\left(\frac{r^2}{n}|{\cal H}|-\sigma_r(\mathcal{H})\right)\left(|{\cal H}|-\binom{n}{r}\right).
   \]
Thus, $|{\cal H}|\geq \frac{n}{r^2}\cdot \sigma_r({\cal H})$. Clearly, equality holds only if all vertex degrees of ${\cal H}$ are equal.
\end{proof}


We use the following bound on the size of regular hypergraphs.

\begin{theorem}[Ihringer and Kupavskii~\cite{ilku}]\label{reg-hg}
    Let $n\geq r$. If $\mathcal{H} \subseteq \binom{[n]}{r}$ is regular and intersecting, then
    \[
    |\mathcal{H}| \leq \frac{1}{1+ \frac{(n-r)(n-r-1)(n-r-2)}{r(r-1)(r-2)}} \binom{n}{r} .
    \]
\end{theorem}

We establish the following simple corollary.

\begin{lemma}\label{reg-lemma}
    For $n\geq 2r+1$, there is no regular intersecting hypergraph $\mathcal{H} \subseteq \binom{[n]}{r}$ whose degrees are $\binom{n-2}{r-2}$.
\end{lemma}

\begin{proof}
    Suppose, to the contrary, that there is such a hypergraph $\mathcal{H}$. Then summing up degrees and applying Theorem~\ref{reg-hg} we have
    \[
    \frac{n}{r}\binom{n-2}{r-2} \leq |\mathcal{H}|\leq \frac{1}{1+ \frac{(n-r)(n-r-1)(n-r-2)}{r(r-1)(r-2)}} \binom{n}{r}.\]
    Rearranging terms gives
\[
\frac{r(r-1)(r-2)+(n-r)(n-r-1)(n-r-2)}{r(r-1)(r-2)} \leq \frac{n-1}{r-1}.
\]
This simplifies to
  \[
(n-r)(n-r-1)(n-r-2) \leq (n-1)r(r-2) - r(r-1)(r-2) = (n-r)r(r-2),
\]  
    giving
    \[
    (n-r-1)(n-r-2) \leq r(r-2),
    \]
    which implies $n < 2r+1$, a contradiction.  
\end{proof}

 The classical Hilton--Milner theorem bounds the size of the largest non-trivial intersecting hypergraph.
\begin{theorem}[Hilton and Milner~\cite{HM}]\label{HM-thm}
Let $r \geq 3$ and $n\geq 2r$.
    If $\mathcal{H} \subseteq \binom{[n]}{r}$ is a non-trivial intersecting hypergraph, then
    \[
    |\mathcal{H}| \leq \binom{n-1}{r-1}  - \binom{n-r-1}{r-1} +1.
    \]
\end{theorem}

The following result of Frankl~\cite{maximum degree}  bounds the size of an intersecting hypergraph under a maximum degree condition.

\begin{theorem}[Frankl~\cite{maximum degree}]\label{maximum degree}
 Let $n> 2r$. If ${\cal H}\subseteq{[n] \choose r}$ is an intersecting hypergraph such that 
 \[
 \Delta({\cal H})\leq {n-1 \choose r-1}-{n-i-1 \choose r-1},
 \]
 for some  $2\leq i\leq r$, then
 \[
 |{\cal H}|\leq {n-1 \choose r-1}-{n-i-1 \choose r-1}+{n-i-1 \choose r-i}.
 \]
\end{theorem}

The size of the ``third largest'' intersecting hypergraph can be bounded as follows.

\begin{theorem}[Han and Kohayakawa~\cite{hako}]\label{3rd-largest-family}
    Let $r \geq 3$ and $n >2r$. If $\mathcal{H} \subseteq \binom{[n]}{r}$ is an intersecting hypergraph that is not a subhypergraph of a $1$-star or of the Hilton--Milner construction $\mathcal{HM}_{n,r}$, then
    \[
    |\mathcal{H}| \leq \binom{n-1}{r-1} - \binom{n-r-1}{r-1} - \binom{n-r-2}{r-2} + 2.
    \]
\end{theorem}

The following theorem bounds the size of an intersecting hypergraph whose edges must intersect three specified sets.

\begin{theorem}[Frankl, Han, Hao, and Yi~\cite{FHHY}]\label{3-matching}
 Let $ r \geq 3, \ell \geq 4, $ and $n \geq r \ell $. Let $ T_1, T_2, T_3 $ be three pairwise disjoint $\ell$-subsets of $[n]$. If ${\cal H} \subseteq \binom{[n]}{r}$ is intersecting and every edge of $\mathcal{H}$ intersects each of $ T_1, T_2, T_3$, then $|{\cal H}| \leq \ell^2 \binom{n-3}{r-3}$.
\end{theorem}

A hypergraph is {\it $t$-intersecting} if every pair of edges have at intersection of size at least $t$.

\begin{theorem}[Wilson~\cite{Wilson}]\label{t-intersecting}
  Let $n \geq (t + 1)(r - t + 1)$. If ${\cal H} \subseteq \binom{[n]}{r}$ is $t$-intersecting, then $|{\cal H}| \leq \binom{n-t}{r-t}$.
\end{theorem}

The following lemma gives conditions to guarantee the existence of a matching of size $s$ containing specified vertices.

\begin{lemma}[Huang, Loh, and Sudakov~\cite{HLS}]\label{HLS}
Let $n>rs$.
If  ${\cal  H} \subseteq \binom{[n]}{r}$ has $s$ distinct vertices $v_1, v_2,\ldots, v_s$ with degrees $\deg(v_i)>2(s-1){n-2 \choose r-2}$, then ${\cal  H}$ contains a 
matching  $e_1,\ldots, e_s$ such that $v_i \in e_i$ for all $i$.
\end{lemma}

Finally, we need the following bound for cross-intersecting non-trivial hypergraphs.

\begin{theorem}[Frankl and Kupavskii~\cite{frku}, Wu and Xiong~\cite{wuxi}]\label{non-triv-cross}
Let $n \geq 2r+1$. If $\mathcal{A},\mathcal{B} \subseteq \binom{[n]}{r}$ are cross-intersecting such that $\bigcap_{F \in \mathcal{A} \cup \mathcal{B}} F = \emptyset$, then
\[
|\mathcal{A}|\cdot |\mathcal{B}| \leq \left(\binom{n-1}{r-1}+1\right)\cdot \left(\binom{n-1}{r-1}-\binom{n-r-1}{r-1}\right).
\]
\end{theorem}

\section{Proof of Theorem~\ref{erdosko}}

\begin{proof}[Proof of Theorem~\ref{erdosko}.]
Assume to the contrary that there is an intersecting $r$-uniform  hypergraph  ${\cal  H}$ on vertex set $[n]$ that has $\sigma_r(\mathcal{H}) \geq r \binom{n-2}{r-2}$, but is
not a $1$-star. 
Applying Lemma~\ref{lem}, we have
\begin{equation}\label{1}
|\mathcal{H}|\geq \frac{n}{r}{n-2 \choose r-2}.
\end{equation}
Note that in~\eqref{1}
 we must have strict inequality, as otherwise Lemma~\ref{lem} implies that $\mathcal{H}$ is a regular intersecting hypergraph with degree $\binom{n-2}{r-2}$ which violates Lemma~\ref{reg-lemma}.

Now, since $\mathcal{H}$ is a non-trivial intersecting hypergraph, Theorem~\ref{HM-thm} implies
$|\mathcal{H}|\le
{n-1 \choose r-1}-{n-r-1 \choose r-1}+ 1$. 
Moreover, for every vertex $x\in [n]$ there is some edge $H\in \mathcal{H}$ such that $x\notin H$, and so $\deg(x)\le \Delta({\cal H}) \le {n-1 \choose r-1}-{n-r-1 \choose r-1}$.  Let $i$ be the smallest positive integer such that
\begin{equation}\label{2}
\Delta({\cal H})\leq {n-1 \choose r-1}-{n-i-1 \choose r-1}.
\end{equation}
Note that necessarily $2 \leq i \leq r$ as otherwise $\Delta(\mathcal{H}) \leq \binom{n-2}{r-2}$ which would imply that we have equality in \eqref{1}.

First suppose that \eqref{2} holds for $i = 2$. Then, using  ${n-1\choose r-1}={n-3 \choose r-1}+2{n-3 \choose r-2}+{n-3 \choose r-3}$ and Theorem~\ref{maximum degree}, we obtain 
\begin{align*}
|{\cal H}|&\leq  {n-1 \choose r-1}-{n-3 \choose r-1}+{n-3 \choose r-2}
\ =\ 3{n-3 \choose r-2}+{n-3 \choose r-3}\\
& ={n-2 \choose r-2}+2{n-3 \choose r-2} 
\ =\ \left(1+\frac{2(n-r)}{n-2}\right){n-2 \choose r-2}.
\end{align*}
Combining with (\ref{1}) (again using that the inequality must be strict), we get $1+\frac{2(n-r)}{n-2}> \frac{n}{r}$. This simplifies to 
\[
(n-(2r+2))(n-r)<  0,
\]
which implies $n<  2r+2$, a contradiction. 

Now, suppose that (\ref{2}) holds for some  $3\leq i\leq r$. Since $i$ is the smallest positive integer such that (\ref{2})  holds, we have
\begin{equation}\label{3}
{n-1 \choose r-1}-{n-i \choose r-1}<\Delta({\cal H})\leq {n-1 \choose r-1}-{n-i-1 \choose r-1}.
\end{equation}

First we give a lower bound on the number of edges of 
$\mathcal{H}$.

\begin{claim}
    $r|\mathcal{H}| > \binom{n-1}{r-1} - \binom{n-i}{r-1} + (n-1)\binom{n-2}{r-2}.$
\end{claim}
\renewcommand{\qedsymbol}{$\blacksquare$}
\begin{proof}
Let $x$ be a vertex of ${\cal H}$ of maximum degree. 
Define the \emph{link hypergraph} ${\cal  H}(x)=\{H\setminus \{x\} \mid  x\in H\in {\cal  H}\}.$
By \eqref{3}, 
\begin{equation*}|\mathcal{H}(x)| > \binom{n-1}{r-1} - \binom{n-i}{r-1} \geq \binom{n-2}{r-1}+\binom{n-2}{r-2} - \binom{n-4}{r-1} > \binom{n-2}{r-2},
\end{equation*} 
so we may apply the Erd\H os--Ko--Rado theorem to find disjoint edges $H_1,H_2 \in \mathcal{H}(x)$.
Put $T:= [n] \setminus (\{x\} \cup H_1 \cup H_2)$ and let $y$ be a vertex of minimum degree in $T$. For each $y'\in T$, the $r$-sets $H_1 \cup \{y'\}$ and $H_2 \cup \{y'\}$ are not in $\mathcal{H}$ as they are disjoint from edges  $H_2 \cup \{x\}$ and $H_1 \cup\{x\}$, respectively. Therefore, $\deg(H_j \cup \{y'\}) \geq r \binom{n-2}{r-2}$ for $j=1,2$, by the Ore-degree condition on $\mathcal{H}$.
We distinguish two cases based on the degree of $y$.

\medskip
\noindent
{\bf Case 1:} $\deg(y) \geq \binom{n-2}{r-2}$.

\medskip
 By the choice of $y$, the degree of each vertex in $T$ is at least $\binom{n-2}{r-2}$. Choose a vertex $y'\in T$, $y'\neq y$. Now, 
\begin{align*}
    r|\mathcal{H}| &= \sum_{z \in [n]} \deg(z) = \deg(x) + \deg(H_1 \cup \{y\}) + \deg(H_2 \cup \{y'\}) + \sum_{z \in T\setminus\{y,y'\}} \deg(z) \\
    & > \binom{n-1}{r-1} - \binom{n-i}{r-1} + 2r\binom{n-2}{r-2} + (n-2r-1) \binom{n-2}{r-2}\\
    & =\binom{n-1}{r-1} - \binom{n-i}{r-1} + (n-1)\binom{n-2}{r-2},
\end{align*}
and we are done.

\medskip
\noindent
{\bf Case 2:} $\deg(y) < \binom{n-2}{r-2}$.

\medskip

Since $\deg(H_1 \cup \{y\}) \geq r\binom{n-2}{r-2}$, we have
\begin{equation}\label{part2}
\sum_{z\in H_1}\deg(z)> (r-1){n-2 \choose r-2}.
\end{equation}
Now, put  $H := H_1 \cup \{x\} \in \mathcal{H}$.
As $\mathcal{H}$ is intersecting, each $r$-set in $V(\mathcal{H}) \setminus H$ is a non-edge, so 
\begin{equation}\label{part3}
     \sum_{z \in V(\mathcal{H})\setminus H} \deg(z) = \frac{1}{\binom{n-r-1}{r-1}}  \left(\sum_{S \in \binom{V(\mathcal{H})\setminus H}{r}} \deg(S) \right) \geq   \frac{\binom{n-r}{r}}{\binom{n-r-1}{r-1}} \cdot r \binom{n-2}{r-2} = (n-r) \binom{n-2}{r-2}.
\end{equation}
Therefore, by \eqref{3}, \eqref{part2}, and \eqref{part3}, we obtain
\begin{align*}
   & r|{\cal H}|  =\sum_{z\in V(\mathcal{H})}\deg(z)=\deg(x)+\sum_{z\in H\setminus \{x\}}\deg(z)+\sum_{z\in V(\mathcal{H})\setminus H}\deg(z)\\
&> {n-1 \choose r-1}-{n-i \choose r-1}+(r-1){n-2 \choose r-2}+(n-r){n-2 \choose r-2}
 = {n-1 \choose r-1}-{n-i \choose r-1}+(n-1){n-2 \choose r-2},
\end{align*}
as desired.
\end{proof}
\renewcommand{\qedsymbol}{$\square$}
 Now, by the Claim and
 using  $(r-1){n-1 \choose r-1}=(n-1){n-2 \choose r-2}$,
we have 
\begin{equation}\label{edgeb}
r|{\cal H}|>{n-1 \choose r-1}-{n-i \choose r-1}+(r-1){n-1 \choose r-1} = r{n-1 \choose r-1}-{n-i \choose r-1}.
\end{equation}
By Theorem~\ref{maximum degree}, $|{\cal H}|\leq {n-1 \choose r-1}-{n-i-1 \choose r-1}+{n-i-1 \choose r-i}$.  Combining this with \eqref{edgeb}  gives
\begin{align}\label{5}
    r{n-i-1 \choose r-1}-{n-i \choose r-1} <    r{n-i-1 \choose r-i}. 
\end{align}
The left-hand side of \eqref{5}  can be rewritten as
\[
 r{n-i-1 \choose r-1}-{n-i \choose r-1} =(r-1){n-i-1 \choose r-1}-{n-i-1 \choose r-2} = \frac{(r-1)(n-i-r)}{n-i-r+1}{n-i-1 \choose r-1},
\]
hence \eqref{5} implies
\begin{equation*}
   \frac{(r-1)(n-i-r)}{n-i-r+1}< r\cdot \frac{{n-i-1 \choose r-i}}{{n-i-1 \choose r-1}}=
   r\cdot \frac{\prod_{j=1}^{n-r-1} \frac{n-i-j}{j}}{\prod_{j=1}^{r-1} \frac{n-i-j}{j}}
   = r\cdot 
 \prod_{j=r}^{n-r-1} \frac{n-i-j}{j}.
\end{equation*}
Therefore, 
\begin{equation*}
    \frac{(r-1)(n-i-r)}{n-i-r+1}<r\cdot\frac{(n-i-r)(n-i-r-1)\cdots(r+2-i)(r+1-i)}{(n-r-1)(n-r-2)\cdots (r+1)r},
    \end{equation*}
    and equivalently, 
\begin{align}\label{77}
 \frac{r-1}{n-i-r+1}<\frac{(n-i-r-1)\cdots(r+1-i)}{(n-r-1)(n-r-2)\cdots(r+1)}.
\end{align}
Clearly, in the range $3 \le i \le r$, the right-hand side of~\eqref{77}  is maximized when $i=3$, also for $3\le i\leq r< n-r$ we have $\frac{r-1}{n-r-2} \leq \frac{r-1}{n-i-r+1}$. Therefore, from \eqref{77} we get
the following inequality:
\begin{align*}
  \frac{r-1}{n-r-2} < \frac{(n-r-4)(n-r-5)\cdots(r-1)(r-2)}{(n-r-1)(n-r-2)\cdots (r+2)(r+1)} = \frac{r(r-1)(r-2)}{(n-r-1)(n-r-2)(n-r-3)},
\end{align*}
which gives
$(n-1-r)(n-r-3) <r(r-2)$, which contradicts the assumption $n\ge 2r+1$. 
\end{proof}

\noindent
{\bf Remark.}
The Hilton--Milner theorem provides a simple proof of Theorem~\ref{erdosko} with the weaker condition $n\geq r^2$. To see this, suppose to the contrary that ${\cal  H}$ is an intersecting $r$-uniform hypergraph on vertex set $[n]$ that is not a $1$-star. By Lemma~\ref{lem},  $|{\cal H}|\geq \frac{n}{r} \binom{n-2}{r-2}$. On the other hand, the Hilton--Milner theorem gives  $|{\cal  H}|< r{n-2 \choose r-2}$, for $n\geq 2r+1$. Therefore, $\frac{n}{r}{n-2 \choose r-2}\leq |{\cal  H}|< r{n-2 \choose r-2}$ which implies $n<r^2$, a contradiction.

\section{Proof of Theorem~\ref{main3}}

The following Proposition, which is in general weaker than Theorem~\ref{main3}, does imply it for $3\leq r \leq 5$.

\begin{proposition}\label{small-r-prop}
Let $n \geq r^2(r-1)$. If $\mathcal{H} \subseteq \binom{[n]}{r}$ is a non-trivial intersecting hypergraph, then $\sigma_r(\mathcal{H}) \leq r \left( \binom{n-2}{r-2} - \binom{n-r-2}{r-2}\right)$.
\end{proposition}

\begin{proof}
   For a contradiction assume that $\mathcal{H} \subseteq \binom{[n]}{r}$ is a non-trivial intersecting hypergraph, where $\sigma_r(\mathcal{H}) > r  \left( \binom{n-2}{r-2} - \binom{n-r-2}{r-2}\right)$.
Fix an arbitrary edge $e$ of $\mathcal{H}$. As $\mathcal{H}$ is non-trivial, for each vertex $x\in e$, there is an edge $e_x$ that does not contain $x$. Define  $T :=e  \cup_{x \in e} e_x$,  the vertex set covered by the edges $e$ and  $e_x$,  where $x\in e$.
We claim that $|T| \leq r^2-r+1$. First, suppose that each of the  edges $e_x$ intersect $e$ in at least two vertices. Then, clearly $|T| \leq r + r(r-2) = r^2-r$. 
Now assume that  there is an $e_x$ such that $e \cap e_x = \{z\}$. The edge $e_z$ does not contain $z$, but must intersect both $e$ and $e_x$, hence the union  of  those three edges includes at most $r+r-1+r-2 = 3r-3$ vertices. Every other  edge $e_y$ (for $y \neq x,z$) must intersect $e \cup e_x \cup e_z$ in at least two vertices, hence  each $e_y$ contributes at most $r-2$ new vertices, for a total of at most $3r-3 + (r-2)(r-2) = r^2-r+1$ vertices, as desired.

Now, observe that every edge of $\mathcal{H}$ must intersect $T$ in at least two vertices,
   hence every $r$-subset $S$ of $[n] \setminus T$ is a non-edge and thus has Ore-degree $\deg(S) > r  \left( \binom{n-2}{r-2} - \binom{n-r-2}{r-2}\right)$.
Averaging over all $r$-subsets of $[n]\setminus T$ gives that the sum of the degrees of vertices in $[n]\setminus T$ is at least
\[
\frac{\binom{n-|T|}{r}}{\binom{n-|T|-1}{r-1}}\cdot r  \left( \binom{n-2}{r-2} - \binom{n-r-2}{r-2}\right) \geq (n-r^2+r-1)\cdot \left( \binom{n-2}{r-2} - \binom{n-r-2}{r-2}\right).
\] 
Now, using that every edge must intersect $T$ in at least two vertices (and that $T$ spans $r+1 \geq 2$ edges), we can give a lower bound on   the number of edges of $\mathcal{H}$ as
    \[
     \frac{n-r^2+r-1}{r-2}  \left( \binom{n-2}{r-2} - \binom{n-r-2}{r-2}\right) + 2 \leq   |\mathcal{H}|.
    \]
On the other hand, recall that the Hilton--Milner construction $\mathcal{HM}_{n,r}$  has Ore-degree  $r \left( \binom{n-2}{r-2} - \binom{n-r-2}{r-2}\right)$, hence we may also assume that $\mathcal{H} \not \subseteq \mathcal{HM}_{n,r}$. Therefore, by Theorem~\ref{3rd-largest-family},
\[
|\mathcal{H}| \leq \binom{n-1}{r-1} - \binom{n-r-1}{r-1} - \binom{n-r-2}{r-2} + 2 = \frac{n-1}{r-1}\binom{n-2}{r-2} - \frac{n-2}{r-1}\binom{n-r-2}{r-2} +2.
\]
Comparing the two bounds on $|\mathcal{H}|$ and rearranging terms gives
\[
\left(\frac{n-r^2+r-1}{r-2} - \frac{n-1}{r-1} \right)\cdot \binom{n-2}{r-2}\leq \left(\frac{n-r^2+r-1}{r-2}-\frac{n-2}{r-1}  \right)\cdot  \binom{n-r-2}{r-2}.
\]
Simplifying and rearranging gives
\[
\frac{\binom{n-2}{r-2}}{\binom{n-r-2}{r-2}} \leq  \frac{n-r^3+2r^2-3}{n-r^3+2r^2-r-1} = 1 + \frac{r-2}{n-r^3+2r^2-r-1}.
\]
Now using Bernoulli's inequality $(1+x)^k \geq 1+kx$ when $x \geq -1$, we get
\[
\frac{\binom{n-2}{r-2}}{\binom{n-r-2}{r-2}} 
 = \frac{(n-2)\cdot\ldots\cdot (n-r+1)}{(n-r-2)\cdot\ldots\cdot (n-2r+1)} \geq \left(\frac{n-2}{n-r-2}\right)^{r-2} = \left(1+ \frac{r}{n-r-2}\right)^{r-2} \ge  1+ \frac{(r-2)r}{n-r-2}.
\]
Combining these two estimates on $\binom{n-2}{r-2}/{\binom{n-r-2}{r-2}}$ gives
\[
\frac{r}{n-r-2} \leq \frac{1}{n-r^3+2r^2-r-1}
\]
which implies $n < r^2(r-1)$, a contradiction.
\end{proof}

\begin{proof}[Proof of Theorem~\ref{main3}.]
By Proposition~\ref{small-r-prop}, we only need to handle the case when $r \geq 6$ 
and $n \geq 4r^2$. Suppose, for the sake of a contradiction, that  $\mathcal{H} \subseteq \binom{[n]}{r}$ is a non-trivial intersecting hypergraph, with $\sigma_r({\cal H})> r\left({n-2 \choose r-2}-{n-r-2 \choose r-2}\right)$. Let $u \in V({\cal H})$ be a vertex of maximum degree $\Delta : = \Delta(\mathcal{H})$. Let ${\cal H}_1 \subseteq \mathcal{H}$ be the subhypergraph formed by of all edges containing vertex $u$. Let ${\cal H}_2 = {\cal H} \setminus {\cal H}_1$ be hypergraph formed by the remaining edges. We consider the following cases.
\medskip

\noindent
{\bf Case 1:} $\Delta \leq \binom{n-1}{r-1} - \binom{n-3}{r-1}$.

\medskip
By Theorem~\ref{maximum degree},
\[
|{\cal H}| \leq \binom{n-1}{r-1} - \binom{n-3}{r-1} + \binom{n-3}{r-2} = \frac{3n - 2r - 2}{n - 2} \binom{n-2}{r-2} \leq 3 \binom{n-2}{r-2}.
\]
Since $\sigma_r({\cal H}) > r\left(\binom{n-2}{r-2} - \binom{n-r-2}{r-2}\right)$, by Lemma~\ref{lem} and~\eqref{handy-ineq}, 
we have  
\begin{align}\label{8}
|{\cal H}| > \frac{n}{r} \left( \binom{n-2}{r-2} - \binom{n-r-2}{r-2} \right) > n \binom{n-r-2}{r-3}. 
\end{align}
Comparing the binomial coefficients  using  $(1+x)^k \geq 1+kx$ and $n\geq 4r^2$ gives
\begin{align*}
  \frac{{n-r-2 \choose r-3}}{{n-3 \choose r-3}}=\frac{(n-r-2)\cdots (n-2r+2)}{(n-3)\cdots (n-r+1)}\geq \Big(1-\frac{r-1}{n-r+1}\Big)^{r-3}
   \geq 1-\frac{(r-1)(r-3)}{n-r+1} \geq 1-\frac{1}{4} =\frac{3}{4}.
\end{align*}
Combining with \eqref{8} yields
\[
|\mathcal{H}|  \ge  \frac{3n}{4} \binom{n-3}{r-3}\ge \frac{3}{4}(r-2)\binom{n-2}{r-2}.
\]

   Now, the upper and lower bounds for $|{\cal H}|$ together imply that $r < 6$, a contradiction.

\medskip

\noindent
{\bf Case 2:} $\Delta > \binom{n-1}{r-1} - \binom{n-3}{r-1}$.

\medskip
By Theorem~\ref{F3}, the $(r-1)$-uniform hypergraph $\{H \setminus \{u\} \mid H \in {\cal H}_1\}$ contains a matching $M = \{T_1, T_2, T_3\}$ of size 3. Every member of ${\cal H}_2$ must intersect each of $T_1, T_2, T_3$.
Thus, by Theorem~\ref{3-matching} 
we have
\begin{equation}\label{7}
|{\cal H}_2| \leq (r - 1)^2 \binom{n-4}{r-3}.
\end{equation}
 Now, we prove  a lower bound for $|{\cal H}_2|$.

\begin{claim}
     $|{\cal H}_2| > \frac{r-2}{2} \binom{n-2}{r-2}$.
\end{claim}

\renewcommand{\qedsymbol}{$\blacksquare$}
\begin{proof}
Clearly, $|{\cal H}_2| \geq 2$, since otherwise ${\cal H}$ is a subhypergraph of the Hilton--Milner construction ${\cal HM}_{n,r}$, and thus $\sigma_r({\cal H}) \leq \sigma_r({\cal HM}_{n,r}) \leq r \left(\binom{n-2}{r-2} - \binom{n-r-2}{r-2}\right)$, a contradiction.  

Let $t = \min |E \cap E'|$ among all distinct $E, E' \in {\cal H}_2$. Obviously $1 \leq t \leq r-1$, and ${\cal H}_2$ is a $t$-intersecting hypergraph. Then since $n \geq 4r^2 \geq (r-t+1)(t+1)+1$, we get $|{\cal H}_2| \leq \binom{n-t-1}{r-t}$ by Theorem~\ref{t-intersecting}. Let $E, E' \in {\cal H}_2$ be such that $|E \cap E'| = t$ and set $S=E \cup E' \cup \{u\}$. 

Observe that $|E|=|E'|=r$ and $|E \cup E'| = 2r-t$.
 Every edge in ${\cal H}_1$ must intersect both $E$ and $E'$. For $x \notin S$, there are $\binom{n-2}{r-2}$ $r$-sets that contain $u$ and $x$. By the principle of inclusion-exclusion, $\binom{n-|E|-2}{r-2} + \binom{n-|E'|-2}{r-2} - \binom{n-|E\cup E'|-2}{r-2}$ of those $r$-sets fail to intersect at least one of $E$ or $E'$. Therefore,
\begin{align}
\label{degreebound1} |{\cal H}_1(x)| \leq \binom{n-2}{r-2} - 2 \binom{n-r-2}{r-2} + \binom{n-2r+t-2}{r-2}.
\end{align}
Let $X = V({\cal H}) \setminus S$ and thus $|X| = n - 1 - (2r - t)\ge r$. As $\mathcal{H}$ is intersecting, there is no edge contained in $X$ and a standard  averaging  over link hypergraphsgraphs $\mathcal{H}(x)$  gives
\begin{equation}\label{0}
\sum_{x\in X}|{\cal  H}(x)|\geq |X|\cdot \left({n-2 \choose r-2}-{n-r-2 \choose r-2}\right).
\end{equation}
Since $\sum_{x\in X}|{\cal H}(x)| = \sum_{x\in X}|{\cal H}_1(x)| + \sum_{x\in X}|{\cal H}_2(x)|$,  applying
\eqref{degreebound1} and \eqref{0} gives
\[
\sum_{x\in X}|{\cal H}_2(x)|\geq|X|\cdot \left(\binom{n-r-2}{r-2} - \binom{n-2r+t-2}{r-2}\right).
\]
For each $H \in {\cal H}_2$, we have  $|H \cap E_1| \geq t$, therefore  every edge $H \in {\cal H}_2$ is counted at most $r - t$ times in $\sum_{x\in X}|{\cal H}_2(x)|$. Therefore,
\[
|{\cal H}_2|\geq \frac{|X|}{r-t} \left( \binom{n-r-2}{r-2}-\binom{n-2r+t-2}{r-2} \right) > |X|\binom{n-2r+t-2}{r-3},
\]
where the last inequality follows by an application of~\eqref{handy-ineq}. 
Substituting $|X|=n-2r+t-1$, we obtain 
\[
|{\cal H}_2| > (r-2)\cdot \binom{n-2r+t-1}{r-2}.
\]
Note that for $n \geq 4r^2$ and $1 \leq t \leq r-1$, we have
\begin{equation*}
    \frac{\binom{n-2r+t-1}{r-2}}{\binom{n-t-1}{r-2}} = \frac{(n-2r+t-1) \cdot\ldots\cdot (n-3r+t+2)}{(n-t-1) \cdot\ldots\cdot (n-t-r+2)} 
 \geq \left( 1 - \frac{2r-2t}{n-t-r+2} \right)^{r-2}  
\geq 1-\frac{2(r-t)(r-2)}{n-t-r+2} 
\geq \frac{1}{2}.
\end{equation*}
Thus,
\[
|{\cal H}_2| > \frac{r - 2}{2} \cdot  \binom{n-t-1}{r-2} \geq \binom{n-t-1}{r-2},
\]
which, together with $|{\cal H}_2| \leq \binom{n-t-1}{r-t}$, implies that $t = 1$, hence $|{\cal H}_2| > \frac{r - 2}{2} \binom{n-2}{r-2}$, as desired.
\end{proof}
\renewcommand{\qedsymbol}{$\square$}

Now, by the Claim, 
$$|{\cal H}_2| > \frac{r-2}{2}\cdot
\binom{n-2}{r-2}
= \frac{r-2}{2}\cdot \frac{n-2}{r-2} \cdot \binom{n-3}{r-3}=\frac{n-2}{2} \cdot\binom{n-3}{r-3}>2(r-1)^2\cdot\binom{n-3}{r-3},
$$
when $n\geq 4r^2\geq 4(r-1)^2+2$, which contradicts \eqref{7}.
\end{proof}

\section{Proof of Theorem~\ref{mainn}}

 We begin with proof for the graph case ($r=2$) of Theorem~\ref{mainn}.

\begin{lemma}\label{graphcase}
If $H$ is an $n$-vertex graph with $\sigma_2(H)>2(s-1)$ and  $n \geq 2s$, then 
$H$ has a matching of size $s$.
\end{lemma}

\begin{proof}
Adding edges between non-adjacent vertices cannot decrease $\sigma_2(H)$, so we may suppose $H$ is edge-maximal subject to containing no matching of size $s$.
If $H$ is a complete graph, then, as $n \geq 2s$, there is a matching of size $s$, a contradiction.
  Let $u,v$ be non-adjacent vertices in $H$. The edge-maximality of $H$ implies that $H+uv$ contains a matching of size $s$. Thus, $H$ has a matching $M$ of size $s-1$. Denote its edges by $E(M)=\{u_1v_1, u_2v_2, \ldots, u_{s-1}v_{s-1}\}$. Let $u,v$ be two vertices not  in $V(M)$. As $H$ has no matching of size $s$, the vertices $u,v$ are not adjacent and each edge incident to $u$ or $v$ is incident to a vertex from $u_1,u_2, \ldots ,u_{s-1},v_1,v_2, \ldots ,v_{s-1}$. Moreover, for each edge $u_iv_i\in M$, the number of edges between $\{u_i,v_i\}$ and $\{u,v\}$ is at most 2. Hence $\deg(u)+\deg(v)\leq 2(s-1) < \sigma_2(H)$, a contradiction.
\end{proof}

\begin{proof}[Proof of Theorem~\ref{mainn}.]
Assume to the contrary that there exists 
an $r$-uniform hypergraph ${\cal  H}$ on $n$ vertices such that $\sigma_r({\cal H})>r\left({n-1 \choose r-1}-{n-s \choose r-1}\right)$, but $\mathcal{H}$ has no matching of size $s$. Without loss of generality we may assume:

\smallskip

\noindent(a)~${\cal H}$ has the minimum number, but at least $3r^2(s-1)$, vertices. 

\smallskip

\noindent (b)~Subject to (a), ${\cal H}$ has the maximum possible number of edges. 

\smallskip

First, note that by Lemma~\ref{graphcase} and Theorem~\ref{erdosko} we may restrict our attention to the case when $r\geq 3$ and $s\geq 3$. 
We first show that the maximum degree of ${\cal  H}$ is at most $r(s-1){n-2 \choose r-2}$. Indeed, suppose $v$ is a vertex with $\deg(v)>r(s-1){n-2 \choose r-2}$. Define the hypergraph $\mathcal{H}'$, which results from the removal of vertex $v$ from $\mathcal{H}$. Because every vertex $u\in V({\cal  H}')$, and $v$ are contained in at most ${n-2 \choose r-2}$ edges of ${\cal  H}$, we have
\[
\sigma_r({\cal H}')>r\left({n-1 \choose r-1}-{n-s \choose r-1}-{n-2 \choose r-2}\right)=r\left({n-2 \choose r-1}-{n-s \choose r-1}\right).
\]
Now, since ${\cal H}$ has the minimum number of vertices and $3r^2\leq n/(s-1)\leq (n-1)/(s-2)$, there is a matching of size $s-1$ in ${\cal H}'$. Each of the $(s-1)r$ vertices in this matching is contained together with $v$ in at most ${n-2 \choose r-2}$ edges. As $\deg(v)>r(s-1){n-2 \choose r-2}$, there exists an edge containing $v$ but disjoint from the $s-1$ edges in the matching; this forms a matching of size $s$ in ${\cal H}$, a contradiction.
Therefore, $\Delta({\cal H})\leq r(s-1){n-2 \choose r-2}$. 

Let $V(\mathcal{H}) = \{v_1,v_2,\ldots, v_n\}$ be such that $\deg(v_1)\geq \deg(v_2)\geq \ldots \geq \deg(v_n)$. By Lemma~\ref{HLS} and the maximum degree condition on ${\cal H}$, we may assume that $\deg(v_1),\ldots, \deg(v_{s-1})\leq r(s-1){n-2 \choose r-2}$, and $\deg(v_s),\ldots, \deg(v_n)\leq 2(s-1){n-2 \choose r-2}$. 
As $n \geq rs$ and by  the edge maximality assumption (b), $\mathcal{H}$ must contain a matching $\mathcal{M}$ of size $s-1$.
By the condition on the vertex degrees of ${\cal H}$, we deduce that the sum of the degrees of the $(s-1)r$ vertices in $\mathcal{M}$ is at most
\[
    (s-1)\left(r(s-1)\binom{n-2}{r-2}\right)+(s-1)(r-1)\left(2(s-1)\binom{n-2}{r-2}\right)
=(3r-2)(s-1)^{2}\binom{n-2}{r-2}.
\]
Now, by
 Lemma~\ref{lem} and inequality \eqref{handy-ineq}, 
\[
|{\cal  H}|\geq \frac{n}{r}\left({n-1 \choose r-1}-{n-s \choose r-1}\right)>\frac{n}{r}(s-1){n-s \choose r-2}.
\]
Using $(1+x)^k \geq 1+kx$, hwe have
\[
{n-s \choose r-2}/{n-2 \choose r-2}>\left(1-\frac{s-2}{n-r}\right)^{r-2}\geq 1-\frac{(s-2)(r-2)}{n-r}>1-\frac{1}{3r},
\]
because $n-r\geq 3r(s-2)(r-2)$.  
Therefore,
\[
|{\cal  H}|>\frac{n}{r}(s-1)\left(1-\frac{1}{3r}\right){n-2 \choose r-2}>(3r-2)(s-1)^2{n-2 \choose r-2},
\]
provided that $n\geq 3r^2(s-1)$. Now we have that the sum of the degrees of the vertices in $\mathcal{M}$ is smaller than the number of the edges of ${\cal  H}$, so there is an edge of ${\cal  H}$ disjoint from $V(\mathcal{M})$, which together with the edges of $\mathcal{M}$ gives a matching of size $s$ in $\mathcal{H}$, a contradiction. 
\end{proof}

Theorem~\ref{mainn} may have some interesting applications. As an example, we give a `colorful' corollary. Let ${\cal H}_1, {\cal H}_2, \ldots , {\cal H}_c$ be hypergraphs. We write ${\cal H}\rightarrow({\cal H}_1, {\cal H}_2, \ldots , {\cal H}_c)$, if for any partition of the edges of ${\cal H}$ into $c$ color classes  ${\cal G}_1, {\cal G}_2, \ldots, {\cal G}_c$,  there is at least one ${\cal G}_i$ which has a subhypergraph isomorphic to ${\cal H}_i$.

\begin{corollary}
Let $n_1\geq n_2\geq \ldots \geq n_c$ be positive integers, put $N=\sum_{i=1}^{c}(n_i-1)$ and let $n\geq 3r^2N$. If ${\cal  H} \subseteq \binom{[n]}{r}$ is a hypergraph with $\sigma_r({\cal H})>r\left({n-1 \choose r-1}-{n-N-1 \choose r-1}\right)$, then 
\[
{\cal  H}\rightarrow({\cal M}_{n_1}, {\cal M}_{n_2}, \ldots, {\cal M}_{n_c}),
\]
where each $\mathcal{M}_{n_i}$ is a matching of size $n_i$. 
\end{corollary}

\begin{proof}
Consider an arbitrary edge coloring of ${\cal  H}$ by colors $\alpha_1, \alpha_2, \ldots, \alpha_c$. By Theorem~\ref{mainn},
${\cal  H}$ contains a matching of size $N+1$. By the pigeonhole principle, there is a monochromatic matching of size $n_i$ whose edges are colored by $\alpha_i$, for some $1\leq i\leq c$. Therefore, ${\cal  H}\rightarrow({\cal M}_{n_1}, {\cal M}_{n_2}, \ldots, {\cal M}_{n_c}).$
\end{proof}

\section{Proof of Theorem~\ref{properlycolored}}

\begin{proof}[Proof of Theorem~\ref{properlycolored}.]
Assume to the contrary that there exists a counterexample, i.e., there are properly edge-colored $r$-uniform hypergraphs ${\cal H}_{1}, \dots, {\cal H}_{s}$ on $n>3r^2s$ vertices such that $\sigma_r({\cal H}_i)> r\left({n-1 \choose r-1}-{n-s \choose r-1}\right)$, but there is no $s$-rainbow matching in ${\cal H}_{1}, \dots, {\cal H}_{s}$. Among counterexamples consider one with minimum $s$. Clearly, $s\geq 2$. First we prove that for each $1\leq i\leq s$, we have  $\Delta({\cal H}_i)\leq r(s-1){n-2 \choose r-2}+s-1$. Indeed, suppose that some ${\cal H}_{i}$ has a vertex $v$ of degree greater than $r(s-1){n-2 \choose r-2}+s-1$. 
For $j \neq i$, define ${\cal H}_j'$ as the subhypergraph of ${\cal H}_{j}$ induced by the vertex set $V({\cal H}_{j}) \setminus \{v\}$.
Since every vertex $u\in {\cal H}^{\prime}_{j}$ and $v$ are contained in at most ${n-2 \choose r-2}$ edges of ${\cal H}_{j}$, we have
\[
\sigma_r({\cal H}^{\prime}_{j})>r\left({n-1 \choose r-1}-{n-s \choose r-1}-{n-2 \choose r-2}\right)=r\left({(n-1)-1 \choose r-1}-{(n-1)-(s-1) \choose r-1}\right).
\]
Now, $3r^2< n/s\leq (n-1)/(s-1)$, and so by the choice of $s$, there exists an $(s-1)$-rainbow matching $\{e_{j}\}_{j \neq i}$ in $\{{\cal H}'_j\}_{j \neq i}$, which spans $r(s-1)$ vertices. Since the degree of $v$ in ${\cal H}_i$ is greater than $r(s-1){n-2 \choose r-2}+s-1$, there exists an edge $e$ in ${\cal H}_{i}$ which is disjoint from the edges of the $(s-1)$-rainbow matching and is of a different color, which produces a $s$-rainbow matching, a contradiction.  Therefore, for $1\leq i\leq s$ we have  $\Delta({\cal H}_i)\leq r(s-1){n-2 \choose r-2}+s-1$. Now, we distinguish two cases based on the $s$-th largest degree in the ${\cal H}_{i}$'s.

\medskip
\noindent
{\bf Case 1:} For some $1\leq i\leq s,$ the $s$-th largest degree in ${\cal H}_{i}$ is at most $2(s-1)\binom{n-2}{r-2}+s-1$.

\medskip
By the choice of $s$, we know that there exists an $(s-1)$-rainbow matching $\mathcal{M}$ in the hypergraphs $\{{\cal H}_j\}_{j\neq i}$ which spans $r(s-1)$ vertices. Since the $s$-th largest degree of ${\cal H}_{i}$ is at most $2(s-1)\binom{n-2}{r-2}+s-1$, the sum of degrees of the $r(s-1)$ vertices of $V(\mathcal{M})$ in ${\cal H}_{i}$ is at most
\[
(s-1)\left[r(s-1)\binom{n-2}{r-2}+s-1\right]+(s-1)(r-1)\left[2(s-1)\binom{n-2}{r-2}+s-1\right]=(s-1)^{2}\left[ (3r-2)\binom{n-2}{r-2}+r \right].
\]
Using Lemma~\ref{lem} and \eqref{handy-ineq}, we have
\[
|{\cal  H}_i|\geq \frac{n}{r}\left({n-1 \choose r-1}-{n-s \choose r-1}\right)>\frac{n}{r}(s-1){n-s \choose r-2}.
\]
Note that
\[
{n-s \choose r-2}/{n-2 \choose r-2}>\left(1-\frac{s-2}{n-r}\right)^{r-2}\geq 1-\frac{(s-2)(r-2)}{n-r}>1-\frac{1}{3r},
\]
because $n-r\geq 3r(s-2)(r-2)$.  
Using the above and the threshold $n > 3r^2s$ we obtain
\[
|{\cal  H}_i|>(3r-1)(s-1)^2{n-2 \choose r-2}
>(3r-2)(s-1)^{2}\binom{n-2}{r-2}+(s-1)^{2}r+(s-1)\frac{n-r(s-1)}{r},
\]
Comparing with the sum of degrees of the $r(s-1)$ vertices of $V(\mathcal{M})$ in ${\cal H}_{i}$, we see that $\mathcal{H}_i$ has more than $(s-1)\frac{n-r(s-1)}{r}$ edges that are disjoint from $\mathcal{M}$. In $\mathcal{H}_i$ there are at most $\frac{n-r(s-1)}{r}$ edges of each color that are disjoint from the vertices of $V(\mathcal{M})$. Therefore, there is some edge disjoint from $\mathcal{M}$ that is of a color not appearing in $\mathcal{M}$. This produces an $s$-rainbow matching in ${\cal H}_{1}, \dots, {\cal H}_{s}$, a contradiction.

\medskip
\noindent
{\bf Case 2:} For all $1\leq i\leq s$, the hypergraph ${\cal  H}_{i}$ has at least $s$ vertices of degree greater than $2(s-1)\binom{n-2}{r-2}+s-1$.

\medskip
For each $1 \leq i \leq s$, 
select distinct vertices $v_{i}\in V({\cal H}_{i})$, such that the degree of $v_{i}$ in ${\cal  H}_{i}$ is greater than $2(s-1)\binom{n-2}{r-2}+s-1$. For each $i$, consider all the subsets of $V({\cal  H}_{i})\setminus\{v_{1},\ldots,v_{s}\}$ which together with $v_{i}$ form an edge of ${\cal  H}_{i}$. Denote the corresponding $(r-1)$-uniform hypergraph by ${\cal  H}'_{i}$. 
Clearly,
\[
|{\cal  H}^{\prime}_{i}|>2(s-1)\binom{n-2}{r-2}+s-1-(s-1)\binom{n-2}{r-2}>\binom{n-s}{r-1}-\binom{n-2s+1}{r-1}.
\]
For every $i$, let the edge $e \in \mathcal{H}_i'$ inherit the color of the edge $e \cup \{v_i\} \in \mathcal{H}_i$.
Now by applying Theorem~\ref{HLW} to the properly edge-colored $(r-1)$-uniform hypergraphs $\mathcal{H}_1'\dots, \mathcal{H}_s'$ on $n-s$ vertices we obtain an $s$-rainbow matching $e_{1} \in \mathcal{H}'_1, e_2 \in \mathcal{H}'_2, \dots, e_s \in \mathcal{H}'_s$. Taking the edges $e_{i}\cup\{v_{i}\}$ for all $i$, we obtain an $s$-rainbow matching in ${\cal H}_{1},\dots, \mathcal{H}_s$, a contradiction.
\end{proof}

\section{Proof of Theorem~\ref{main2}}

\begin{proof}
    Suppose
$\mathcal{A},\mathcal{B} \subseteq \binom{[n]}{r}$ are cross-intersecting hypergraphs. If the edges of both hypergraphs  share a common vertex, then both are trivial (thus intersecting) and so by Theorem~\ref{erdosko} we have
\[
\sigma_r(\mathcal{A})\cdot \sigma_r(\mathcal{B})\le  r^2 \binom{n-2}{r-2}^2.
\]
If there is no such common vertex, then we may apply Lemma~\ref{lem}, Theorem~\ref{non-triv-cross}, and inequality \eqref{handy-ineq} to obtain
\begin{align*}
    \sigma_r(\mathcal{A})\cdot \sigma_r(\mathcal{B}) \leq \frac{r^4}{n^2} |\mathcal{A}||\mathcal{B}| & \leq 
 \frac{r^4}{n^2} \cdot \left(\binom{n-1}{r-1}+1\right)\cdot \left(\binom{n-1}{r-1}-\binom{n-r-1}{r-1}\right) 
  < \frac{r^4}{n^2}\cdot 2r\cdot \binom{n-1}{r-1}\cdot \binom{n-2}{r-2} \\
  &= \frac{2r^5(n-1)}{(r-1)n^2}\binom{n-2}{r-2}^2
   \leq r^2 \binom{n-2}{r-2}^2,
\end{align*}
whenever $n \geq 4r^2$.
\end{proof}

\section{Concluding remarks}

In this paper we proved Ore-degree analogues of the Erd\H{o}s--Ko--Rado theorem, Hilton--Milner theorem, and the Erd\H{o}s Matching Conjecture when the number of the vertices of the host hypergraph is large enough.
It would be interesting to find the exact conditions on $n$ for which these statements hold.
We would like to emphasize the problem of finding a matching of fixed size.
Recall that the $r$-uniform hypergraph ${\cal H}$ on $n$ vertices containing all the edges intersecting a fixed set of $s-1$ vertices has matching number $s-1$ and Ore-degree $\sigma_r({\cal H})=r\left({n-1 \choose r-1}-{n-s \choose r-1}\right)$. Theorem~\ref{mainn} shows that for $n\geq 3r^2(s-1)$, this bound on $\sigma_r({\cal H})$ is best possible. We expect that the same should hold when $n>Crs$ for some absolute constant $C$.

Note that Theorem~\ref{mainn} cannot be extended to hold for $n\leq rs+1$ in general. Indeed, consider the hypergraph on $n=rs+1$ vertices formed by a clique on $rs-1$ vertices together with two isolated vertices.
This hypergraph contains no matching of size $s$ and has Ore-degree  $(r-2)\binom{n-3}{r-2}$ which exceeds the bound from Theorem~\ref{mainn} for many values of $r$ and $s$ (e.g.\ $r=8$ and $s=7$).
Therefore, we make the following conjecture.

\begin{conjecture}\label{con1}
For $n>2rs$, every $r$-uniform hypergraph ${\cal H}$  on $n$ vertices with  $\sigma_r({\cal H})>r\left({n-1 \choose r-1}-{n-s \choose r-1}\right)$ contains a matching of size $s$.
\end{conjecture}
\noindent 
For $r=2$, a generalization of this conjecture was proved in \cite{ghaffar}, but  it remains open for  $r\geq 3$.\\

 \noindent {\bf Acknowledgments.} We are grateful to Nika Salia and Casey Tompkins for the example showing that Conjecture~\ref{con1} cannot hold for $n\leq rs+1$.


\small
\begin{thebibliography}{10}
\bibitem{Barat}
J. Bar\'{a}t, G. N. S\'{a}rk\"{o}zy, Partitioning 2-edge-colored Ore-type graphs by monochromatic cycles, {\it J. Graph Theory} {\bf 81} (4) (2016), 317--328.

\bibitem{BDE}
B. Bollob\'{a}s, D. E. Daykin, P. Erd\H{o}s, Sets of independent edges of a hypergraph, {\it Q. J. Math. Oxf. Ser.} {\bf27} (105) (1976), 25--32.

\bibitem{erdos}
P. Erd\H{o}s, A problem on independent $r$-tuples, {\it Ann. Univ. Sci. Budapest.} {\bf 8} (1965), 93--95.

\bibitem{erdos-gallai}
P. Erd\H{o}s, T. Gallai, On the minimal number of vertices representing the edges of a graph, {\it Publ. Math. Inst. Hung. Acad. Sci.} {\bf 6} (1961), 181--203.

\bibitem{EKR}
P. Erd\H{o}s, C. Ko, R. Rado, Intersection Theorems for Systems of Finite Sets, {\it Quart. J.
Math.} {\bf 12} (1) (1961), 313--320.


\bibitem{fad}
R. J. Faudree, R. J. Gould, A. V. Kostochka, L. Lesniak, I. Schiermeyer, A. Saito, Degree conditions for $k$-ordered Hamiltonian graphs, {\it J. Graph Theory} {\bf 42} (2003), 199--210.

\bibitem{Fe}
M. Ferrara, R. Gould, M. Jacobson, F. Pfender, J. Powell, T. Whalen, New Ore-type conditions for $H$-linked graphs, {\it J. Graph Theory} {\bf 71} (2012), 69--77.

\bibitem{maximum degree}
P. Frankl, Erd\H{o}s–Ko–Rado theorem with conditions on the maximal degree, {\it J. Combin. Theory A} {\bf 46} (1987), 252--263.

\bibitem{F2}
P. Frankl, Improved bounds for Erd\H{o}s’ matching conjecture, {\it J. Combin. Theory Ser. A} {\bf 120} (2013), 1068--1072.

\bibitem{F}
P. Frankl, On the maximum number of edges in a hypergraph with a given matching number, {\it Discrete Appl. Math.} {\bf 216} (2017), 562--581.

\bibitem{FG}
P. Frankl, R. L. Graham, Old and new proofs of the Erd\H{o}s–Ko–Rado theorem, {\it Sichuan Daxue Xuebao} {\bf 26} (1989), 112--122.

\bibitem{FHHY}
P. Frankl, J. Han, H. Hao, Z. Yi, A degree version of the Hilton–Milner theorem, {\it J. Combin. Theory Ser. A} {\bf 155} (2018), 493--502.


\bibitem{frku} P. Frankl, A. Kupavskii,
A size-sensitive inequality for cross-intersecting families.
{\it European J. Combin.} {\bf 62} (2017), 263--271.

\bibitem{FRR}
P. Frankl, V. R\"{o}dl, A. Ruci\'{n}ski, On the maximum number of edges in a triple system not containing a disjoint family of a given size, {\it Combin. Probab. Comput.} {\bf 21} (2012), 141--148.

\bibitem{FT}
P. Frankl, N. Tokushige, A note on Huang–Zhao theorem on intersecting families with large minimum degree, {\it Discrete Math.} {\bf 340} (5) (2017), 1098--1103.

\bibitem{franklwang}
P. Frankl, J. Wang, On the Largest Degrees in Intersecting Hypergraphs,
arXiv:2511.15508.

\bibitem{hako}
J. Han, Y. Kohayakawa, 
The maximum size of a non-trivial intersecting uniform family that is not a subfamily of the Hilton-Milner family,
{\it Proc. Amer. Math. Soc.} {\bf 145} (1)  (2017), 73--87.


\bibitem{HM}
A. Hilton, E. Milner, Some intersection theorems for systems of finite sets, {\it Q. J. Math. Oxf. Ser.}  {\bf 18} (2) (1967), 369--384.

\bibitem{HLS}
H. Huang, P. Loh, B. Sudakov, The size of a hypergraph and its matching number, {\it Combin. Probab. Comput.} {\bf 21} (2012), 442--450.

\bibitem{huang}
H. Huang, Y. Zhao, Degree versions of the Erd\H{o}s--Ko--Rado theorem and Erd\H{o}s hypergraph matching,  {\it J. Combin. Theory Ser. A} {\bf 150} (2017), 233--247.





\bibitem{HLW}
H. Huang, T. Li, G. Wang, Rainbow matchings in properly-colored hypergraphs, {\it Electron. J. Combin.} {\bf 26} (1) (2019),\#P1.4.

\bibitem{huangrao}
H. Huang, R. Rao, On the $\ell$-th largest degree of an intersecting family, arXiv:2602.01692.

\bibitem{ilku}
F. Ihringer, A. Kupavskii, 
Regular intersecting families,
{\it Discrete Appl. Math.} {\bf 270} (2019), 142--152.

\bibitem{equ}
H. A. Kierstead, A.V. Kostochka, An Ore-type theorem on equitable coloring, {\it J. Combin. Theory Ser. B} {\bf 98} (1) (2008), 226--234.

\bibitem{brooks}
H. A. Kierstead, A. V. Kostochka, Ore-type versions of Brooks theorem, {\it J. Combin. Theory Ser. B} {\bf 99} (2) (2009), 298--305.




\bibitem{Kupavskii} 
A. Kupavskii, Degree versions of theorems on intersecting families
via stability, {\it J. Combin. Theory Ser. A} {\bf 168} (2019), 272--287.

\bibitem{LM}
T. {\L}uczak, K. Mieczkowska, On Erd\H{o}s extremal problem on matchings in hypergraphs, {\it J. Combin. Theory Ser. A} {\bf 124} (2014), 178--194.


\bibitem{ghaffar}
G. R. Omidi, G. Raeisi, Ramsey-Tur\'an type results for matchings in edge colored graphs, {\it Discrete Math.} {\bf 347} (3) (2024), 113785.




\bibitem{ore}
O. Ore, Note on Hamilton circuits, {\it Am. Math. Mon.} {\bf 67} (1960), 55.

\bibitem{Pyber}
L. Pyber, A new generalization of the Erd\H{o}s-Ko-Rado theorem, {\it J. Combin. Theory Ser. A} {\bf 43} (1986), 85--90.

\bibitem{Wilson}
R. M. Wilson, The exact bound in the Erd\H{o}s-Ko-Rado theorem, {\it Combinatorica} {\bf 4} (1984), 247--257.

\bibitem{wuxi} B. Wu, R. Xiong, H.~Rong,
A note on the maximum product-size of non-trivial cross $t$-intersecting families,
{\it Discrete Math.} {\bf 347} (2) (2024), 113783.

\end{thebibliography}
\end{document}